\documentclass[11pt]{amsart}
\usepackage[T1]{fontenc}
\usepackage{amsmath,amssymb,amsthm}
\usepackage{mathtools}
\usepackage[hypertexnames=false]{hyperref}%ないとだめ
\usepackage{cleveref}
\usepackage{autonum}
\usepackage{booktabs}
\usepackage{longtable}
\usepackage{float}
\usepackage[numbers]{natbib}
\newtheorem{thm}{Theorem}[section]

\newtheorem{cor}{Corollary}[section]
\newtheorem{eg}{Example}[section]
\newtheorem{definition}{Definition}[section]

\newtheorem{rem}{Remark}[section]

\newtheorem{prob}{Problem}[section]

\crefname{thm}{Theorem}{Theorem}
\crefname{lem}{Lemma}{Lemma}
\crefname{prop}{Proposition}{Proposition}
\crefname{cor}{Corollary}{Corollary}
\crefname{eg}{Example}{Example}
\crefname{definition}{Definition}{Definition}
\crefname{q}{Question}{Question}
\crefname{rem}{Remark}{Remark}
\crefname{conj}{Conjecture}{Conjecture}
\crefname{claim}{Claim}{Claim}
\crefname{prob}{Problem}{Problem}

\def\fr#1{\mathfrak{#1}}
\def\bb#1{\mathbb{#1}}

\def\ali#1{\begin{align}#1\end{align}}

\def\w{\wedge}
\def\bw{\bigwedge}
\def\o{\oplus}
\def\bo{\bigoplus}
\def\ang#1{\langle #1\rangle}
\def\dim#1{{\rm dim\ }#1}
\def\0{\{ 0\}}
\def\Ker#1{{\rm Ker}\ #1}
\def\Im#1{{\rm Im}\ #1}
\def\bs{\backslash}

\title[Quasi-projective nilmanifolds]{Quasi-projective nilmanifolds}
\author{Taito Shimoji}
\date{}
\address{Department of Mathematics, Graduate School of Science, The University of Osaka, Osaka, Japan}
\email{u215629i@ecs.osaka-u.ac.jp}
\begin{document}
\maketitle
\begin{abstract}
Let $M=V\setminus D$ be a smooth quasi-projective variety for some smooth projective variety $V$ and a divisor $D$ with normal crossings. Assume that $M$ is diffeomorphic to a non-compact nilmanifold $\Gamma\backslash N\times\mathbb{R}^m$. We show that $M$ is diffeomorphic to a trivial bundle $T^n\times \mathbb{R}^m$ over a torus $T^n$ if the first cohomology $H^1(V)$ of $V$ vanishes. Moreover, in general, we show that $M$ is diffeomorphic to a trivial bundle $T^{b_1(M)}\times \mathbb{R}^m$ over a $b_1(M)$-dimensional torus $T^{b_1(M)}$, or a trivial bundle $E\times \mathbb{R}^m$ such that $E$ is a torus bundle $E\rightarrow T^{b_1(M)}$ over a torus $T^{b_1(M)}$. Conversely, we consider whether non-compact nilmanifolds are diffeomorphic to a smooth quasi-projective variety. We determine the Lie groups of dimension up to $8$ such that  corresponding non-compact nilmanifolds may be diffeomorphic to smooth quasi-projective varieties.
\end{abstract}
\section{Introduction}
Let $M=V\setminus D$ be a smooth quasi-projective variety for some smooth projective variety $V$ (which is called a compactification of $M$) and a divisor $D$ with normal crossings. In \cite{Del}, Deligne showed that the complexification $H^*(M)_{\bb{C}}$ of the real cohomology $H^*(M)$ admits the bigrading of a mixed Hodge structure $(H^*(M),F,W)$
\ali{
H^k(M)_{\bb{C}}=\bo_{(p,q)\in I_k}H_{p,q}^k(M)
}
such that
\ali{
\overline{H_{p,q}^k(M)}\equiv H_{q,p}^k(M)\mod\bo_{s+t<p+q}H_{s,t}^k(M) 
}
for all $k\geq 0$ where
\ali{
I_k:=\{ (p,q)\in \bb{Z}\times \bb{Z}\mid k\leq p+q\leq 2k\}.
}
For example, $H^1(M)_{\bb{C}}=H_{1,0}^1(M)\o H_{0,1}^1(M)\o H_{1,1}^1(M)$ (see \cref{MHS} and \cref{Del}).

Let $N$ be a simply-connected nilpotent Lie group and $\Gamma$ a discrete subgroup of $N$. Assume that the quotient space $M=\Gamma\bs N$ is compact. It is well-known that  $M$ is smooth quasi-projective variety if and only if $N$ is abelian Lie group, and $M$ is diffeomorphic to a torus (see \cite{BensonGordon},\cite{Hasegawa} and \cite{Hasegawa2}). 
We consider the case that the quotient space $M=\Gamma\bs N$ is not compact. In this case, we call $M$ non-compact nilmanifold. It is known that every non-compact nilmanifold $M$ can be written as the product $M=\tilde{\Gamma}\bs \tilde{N}\times \bb{R}^m$ for some (compact) nilmanifold $\tilde{\Gamma}\bs \tilde{N}$ and integer $m$ (see \cite{Mal}). We consider the following problem.
\begin{prob}\label{prob}
When is a smooth quasi-projective variety diffeomorphic to a non-compact nilmanifold?
\end{prob}
For example, let us consider $M=\Gamma\bs H_3(\bb{R})\times \bb{R}$ where $H_3(\bb{R})$ is the $3$-dimensional Heisenberg group and $\Gamma$ is a lattice in $H_3(\bb{R})$. We see that $M$ is a principal $\bb{C}^*$-bundle over a torus $T^2$. Defining a complex structure on $T^2$, $M$ is a smooth quasi-projective variety (see \cref{eg} and \cite{Morgan}).
 
We first obtain the following result by using the condition of mixed Hodge structures on the cohomologies of smooth quasi-projective varieties. We define the first Betti number $b_1(M)$ of a smooth quasi-projective variety $M$ by
\ali{
b_1(M):=\dim H^1(M)
}
where $H^1(M)$ is the real cohomology of $M$. Note that $b_1(M)=b_1(\Gamma\bs N)$.
\begin{thm}[\cref{abel}]
Let $M=V\setminus D$ be a smooth quasi-projective variety. Assume that $M$ is diffeomorphic to a non-compact nilmanifold $\Gamma\bs N\times \bb{R}^m$. If $H^1(V)=\0$ for the compactification $V$ of $M$, then $M$ is diffeomorphic to a trivial bundle $T^n\times \bb{R}^m$ over a torus $T^n$ where $n=\dim N$. In particular, $M$ is diffeomorphic to a trivial bundle $T^{b_1(M)}\times \bb{R}^m$ over a $b_1(M)$-dimensional torus.
\end{thm}
In general, we show the following result. The proof is the same argument as in the above result.   
\begin{thm}[\cref{abel2step}]\label{main1}
Let $M$ be a smooth quasi-projective variety. Assume that $M$ is diffeomorphic to a non-compact nilmanifold $\Gamma\bs N\times \bb{R}^m$. Then $M$ is diffeomorphic to a trivial bundle $T^{b_1(M)}\times\bb{R}^m$ over a $b_1(M)$-dimensional torus $T^{b_1(M)}$ or a trivial bundle $E\times \bb{R}^m$ such that $E$ is a torus bundle $E\rightarrow T^{b_1(M)}$ over a torus $T^{b_1(M)}$. In particular, the fundamental group $\pi_1(M,x)$ is torsion-free nilpotent and the nilpotency class is at most two. 
\end{thm}
The proof of the above results follows from the following corollary. The following corollary can be shown by using the results of Nomizu \cite{Nomizu}, Deligne \cite{Del} and Morgan \cite{Morgan}.
\begin{cor}[\cref{exNomizu}]
Let $M$ be a smooth quasi-projective variety. Assume that $M$ is diffeomorphic to a non-compact nilmanifold $\Gamma\bs N\times \bb{R}^m$. Let $\fr{n}$ be the Lie algebra of $N$. Then $\fr{n}_{\bb{C}}$ admits the bigrading $\fr{n}_{\bb{C}}=\bo_{p,q\leq 0,p+q\leq -1}\fr{n}_{p,q}$ of a mixed Hodge structure such that for the induced bigrading $H^*(\fr{n}_{\bb{C}})=\bo_{p,q}H_{p,q}^*(\fr{n}_{\bb{C}})$, the following conditions hold:
\ali{
&\bullet\ H^j(\fr{n}_{\bb{C}})=\bo_{(p,q)\in I_j}H_{p,q}^j(\fr{n}_{\bb{C}})\\
&\bullet\ \overline{H_{p,q}^j(\fr{n}_{\bb{C}})}\equiv H_{q,p}^j(\fr{n}_{\bb{C}}) \mod \bo_{s+t<p+q}H_{s,t}^j(\fr{n}_{\bb{C}})
}
for all $0\leq j\leq \dim\fr{n}_{\bb{C}}$ where $I_j$ is given by
\ali{
I_j:=\{ (p,q)\in \bb{Z}^2\mid j\leq p+q\leq 2j,0\leq p,q\leq j\}.
}
\end{cor}
Conversely, we consider whether a non-compact nilmanifold $\Gamma\bs N\times\bb{R}^m$ is diffeomorphic to a smooth quasi-projective variety. By the proof of \cref{main1} (\cref{abel}), we give a sufficient condition for non-compact nilmanifolds not to be diffeomorphic to smooth quasi-projective varieties. 
\begin{cor}[\cref{3stepNG}]\label{3stepNGintro}
Let $N$ be a simply-connected nilpotent Lie group. Assume that the Lie algebra $\fr{n}=\fr{n}_{\bb{Q}}\otimes\bb{R}$ of $N$ is $(3\leq )t$-step nilpotent or, $2$-step and not isomorphic to $\tilde{\fr{n}}\o \bb{R}^n$ for some $2$-step nilpotent Lie algebra $\tilde{\fr{n}}$ and abelian Lie algebra $\bb{R}^n$ such that 
\ali{
\dim H^1(\fr{n})\notin 2\bb{Z}.
}
Then $M=\Gamma\bs N\times \bb{R}^m$ is not a smooth quasi-projective variety for any lattice $\Gamma$ in $N$ and integer $m>0$.
\end{cor}
Moreover, in \cref{upto7} and \cref{dim8}, we determine the nilpotent Lie groups such that corresponding non-compact nilmanifolds may be diffeomorphic to smooth quasi-projective varieties. We use \cref{3stepNGintro} and the classification of $2$-step nilpotent Lie algebras of dimension up to $8$ in \cite{Nil} and \cite{YanDeng}( or \cite{RenZhu}).
\subsection{Comparing with our previous work}(\cref{Prework})
In \cite{Shimoji}, we consider whether it can be shown that the torsion-free nilpotent fundamental groups of smooth quasi-projective varieties are abelian or $2$-step by using the Morgan's result in \cite[Theorem (9.4)]{Morgan}. We showed that the nilpotency classes of the fundamental groups of smooth quasi-projective varieties of rank up to $7$ are at most $2$ (see \cite[Theorem 1.2]{Shimoji}). Moreover, we determined the Lie groups of dimension up to $7$ whose lattices may be isomorphic to the fundamental groups of smooth quasi-projective varieties (see \cite[Theorem 1.3 and 1.4]{Shimoji}). However, from rank $8$, we cannot determine the nilpotency classes of the fundamental groups of smooth quasi-projective varieties. In this paper, we focus on the class of the intersection between smooth quasi-projective varieties and non-compact nilmanifolds. We can show the restriction of the nilpotency classes of the fundamental groups in arbitrary dimension in this setting. 
\subsection{Acknowledgements}
The author thanks his supervisor, Professor Hisashi Kasuya at Nagoya University, for suggesting that the author consider \cref{prob} and  This work was supported by JST SPRING, Grant Number JPMJSP2138.
\section{Cohomologies of nilmanifolds}
In this section, we first recall the definitio of Differential Graded Algebras (DGA) and Minimal models. A graded algebra $A=\bo_{i\geq 0}A^i$ over a field $\bb{K}$ of characteristic $0$ is called {\em graded-commutative} if 
\ali{
xy=(-1)^{ij}yx\text{ for all }i\in A^i,y\in A^j.
}
A {\em differential} on $A$ is a linear map $d;A\rightarrow A$ such that $d\circ d=0$ and $d(A^i)\subset A^{i+1}$ for all $i\geq 0$. A {\em $\bb{K}$-DGA} $(A,d)$ is a pair of a graded-commutative graded algebra $A$ over a field $\bb{K}$ of characteristic $0$ and a differntial $d$ on $A$ such that the following conditions hold:
\begin{itemize}
\item The $\bb{K}$-algebra structure of $A$ is given by $\bb{R}\hookrightarrow A^0$.
\item The Leibniz rule holds: $d(ab)=d(a)b+(-1)^iad(b)$ for all $a\in A^i,b\in A^j$.
\end{itemize} 
The {\em $i$-th cohomology} of a DGA $(A,d)$ is given by
\ali{
H^i(A,d):=\frac{\Ker (d:A^i\rightarrow A^{i+1})}{\Im (d:A^{i-1}\rightarrow A^i)}.
}
Let $(A,d_A)$ and $(B,d_B)$ be $\bb{K}$-DGAs. A $\bb{K}$-linear map $f:A\rightarrow B$ is called a {\em morphism} of $\bb{K}$-DGA if $f(A^i)\subset B^i$ for all $i\geq 0$ and $d_B\circ f=f\circ d_A$. A morphism of $\bb{K}$-DGA $f:A\rightarrow B$ is called {\em quasi-ismorphism} if the induced map $f^*:H^*(A,d_a)\rightarrow H^*(B,d_B)$ is isomorphism. 
\begin{definition}[\cite{Boc16}]
Let $V=\bo_{i\geq 1}V^i$ be a graded vector space such that $V$ is spanned by a set $\{ a_k\mid k\in I\}$ for some well-ordered set $I$. Assume that $k<l$ implies $|a_k|\leq |a_l|$ where $|a|:=i$ for $a\in V^i$. A $\bb{K}$-DGA $(\bw V,d)$ is a {\em minimal model} of a $\bb{K}$-DGA $(A,d_A)$ if the following conditions hold:
\begin{itemize}
\item $d(a_k)\in \bw\text{Span}\{ a_l\mid l<k\}$ for all $k\in I$.
\item There exists a quasi-isomorphism $f:\bw V\rightarrow A$.  
\end{itemize} 
\end{definition}
\subsection{Nomizu's theorem}
Let $N$ be a simply-connected nilpotent Lie group and $\Gamma$ a lattice (discrete and cocompact subgroup) in $N$. The homogeneous space $M=\Gamma\bs N$ is called nilmanifold. Nomizu showed the following result.
\begin{thm}[\cite{Nomizu}]\label{Nomizu}
Let $M=\Gamma\bs N$ be a nilmanifold. Let $\fr{n}$ be the Lie algebra of $N$. Then the natural inclusion $(\bw^{\bullet}\fr{n}^*,d)\hookrightarrow (A^*(M),d)$ induces an isomorphism
\ali{
H^*(\fr{n})\cong H_{{\rm dR}}^*(M)
}
where $A^*(M)$ consists of left-invariant forms. In particular, $(\bw^{\bullet}\fr{n}^*,d)$ is a minimal model of $(A^*(M),d)$ in the sense of Sullivan (see \cite{Sul}).
\end{thm}
Thus, we can compute de Rham cohomologies of nilmanifolds by Lie algebra cohomologies. \begin{rem}[\cite{Mal},\cite{Raghu}]\label{lattice}
Let $N$ be a simply-connected nilpotent Lie group and $\fr{n}$ the Lie algebra of $N$. Assume $\dim N=n$. Then $N$ admits a lattice if and only if there exists a basis $X_1,\dots,X_n$ such that the structure constants $\{ C_{ij}^k\}$ given by
\ali{
[X_i,X_j]=\sum_{k=1}^nC_{ij}^kX_k
}
are all in $\bb{Q}$.
\end{rem}
\begin{eg}
Let $H_3(\bb{R})$ be the $3$-dimensional Heisenberg group. We see that there exist a basis of the Lie algebra $\fr{n}_3(\bb{R})=\ang{X_1,Y_1,Z}$ of $H_3(\bb{R})$ such that 
\ali{
[X_1,Y_1]=Z,[X_1,Z]=0=[Y_1,Z].
}
Since the structure constants are all in $\bb{Q}$, the Lie group $H_3(\bb{R})$ admits a lattice. For example, $H_3(\bb{Z})$ is a lattice in $H_3(\bb{R})$.
\end{eg}
Let $N$ be a simply-connected nilpotent Lie group and $\Gamma$ a lattice in $N$. The homogeneous space $M=\Gamma\bs N$ is called {\em nilmanifold}.
\begin{eg}
Let $M=H_3(\bb{Z})\bs H_3(\bb{R})$ be a nilmanifold where $H_3(\bb{R})$ is $3$-dimensional Heisenberg group. Let $X_1,Y_1$ and $Z$ be a basis of $H_3(\bb{R})$ as in \cref{lattice}. Then the cohomologies of $\fr{n}_3(\bb{R})$ are given by 
\ali{
H_{{\rm dR}}^0(M)&\cong H^0(\fr{n})=\bb{R},\\
H_{{\rm dR}}^1(M)&\cong H^1(\fr{n})=\ang{[x_1],[y_1]},\\
H_{{\rm dR}}^2(M)&\cong H^2(\fr{n})=\ang{[x_1\w z],[y_1\w z]},\\
H_{{\rm dR}}^3(M)&\cong H^3(\fr{n})=\ang{[x_1\w y_1\w z]},\\
H_{{\rm dR}}^j(M)&\cong H^j(\fr{n})=\0 \text{ for all }j>3\\
}
where $x_1,y_1$ and $z$ are the dual basis of $X_1,Y_1$ and $Z$.
\end{eg}
\section{Deligne's Mixed Hodge structures}
\begin{definition}[\cite{Del},\cite{Morgan},\cite{PeterSteen}]\label{MHS}
For an $\bb{R}$-vector space $V$, an $\bb{R}$-\textit{mixed Hodge structure} $(V,W,F)$ is defined as the following data:
\begin{itemize}
\item $W=\{W_k(V)\}_{k\in\bb{Z}}$ is a finite increasing filtration on $V$.
\item $F=\{F^p(V_{\bb{C}})\}_{p\in\bb{Z}}$ is a finite decreasing filtration on the complexification $V_{\bb{C}}:=V\otimes\bb{C}$ of $V$ and defines a Hodge structure on 
\ali{
Gr_k^W(V_\bb{C}):=\frac{W_k(V)\otimes\bb{C}}{W_{k-1}(V)\otimes\bb{C}}
}
of weight $k$ for all $k\in\bb{Z}$. The decreasing filtration is given by
\ali{
F^p(Gr_k^W(V_\bb{C}))=\frac{F^p(V_{\bb{C}})\cap (W_k\otimes \bb{C})}{W_{k-1}\otimes\bb{C}}.
}
{\it i.e.} there exists a bigrading $Gr_k^W(V_{\bb{C}})=\bo_{p+q=k}\mathcal{H}_{p,q}$ such that
\ali{
\overline{\mathcal{H}_{p,q}}=\mathcal{H}_{q,p}\ \text{ for all }p,q\in\bb{Z}\text{ with }p+q=k.
}
\end{itemize}
\end{definition}
Let $(V,W,F)$ be a mixed Hodge structure. Defining the subspace $V_{p,q}$ of $V_{\bb{C}}$ by
\ali{
V_{p,q}:=F^p(V_{\bb{C}})\cap W_{p+q}(V_{\bb{C}})\cap \left (\overline{F^q(V_{\bb{C}})}\cap W_{p+q}(V_{\bb{C}})+\sum_{i\geq 2}\overline{F^{q-i+1}(V_{\bb{C}})}\cap W_{p+q-i}(V_{\bb{C}})\right)}
where $W_k(V_{\bb{C}}):=W_k\otimes\bb{C}$, we have the bigrading 
\ali{
V_{\bb{C}}=\bo_{p,q}V_{p,q}
}
such that
\ali{
\overline{V_{p,q}}\equiv V_{q,p}&\mod \bo_{s+t<p+q}V_{s,t},\\
W_k(V_{\bb{C}})=&\bo_{p+q\leq k}V_{p,q},
}
and
\ali{
F^p(V_{\bb{C}})&=\bo_{s\geq p,t\in\bb{Z}}V_{s,t}
.}
In \cite{Del}, Deligne showed the following theorem.
\begin{thm}[\cite{Del}]\label{Del}
Let $M$ be an algebraic variety. For all $1\leq j\leq \dim M$, the cohomology $H^j(M)$ admits a mixed Hodge structure $(H^j(M),F,W)$ such that the induced bigrading $H^j(M)_{\bb{C}}=\bo_{p,q\in\bb{Z}}H_{p,q}^j(M)$ is given by
\ali{
H^j(M)_{\bb{C}}=\bo_{(p,q)\in I_j}H_{p,q}^j(M)
}
where $I_j$ is given by
\ali{
I_j:=\{ (p,q)\in \bb{Z}^2\mid j\leq p+q\leq 2j,0\leq p,q\leq j\}.
}
\end{thm}
In \cite{Morgan}, Morgan showed the following result. The result gives a restriction to the fundamental groups of smooth quasi-projective varieties. 
\begin{thm}[\cite{Morgan},Theorem (8.6)]\label{Morgan}
Let $M$ be a smooth quasi-projective variety. Let $(A,d)$ be an $\bb{R}$-DGA. If $(A,d)$ is a minimal model of the complex $(A^{\bullet}(M),d)$, then the complexification $A_{\bb{C}}$ admits the bigrading of a mixed Hodge structure $A_{\bb{C}}=\bo_{p,q,p+q\geq 0}A_{p,q}$ such that the induced bigrading $H^*(A_{\bb{C}},d)=\bo_{p,q,p+q\geq 0}H_{p,q}^*(A_{\bb{C}})$ is compatible with a bigrading of the Deligne's mixed Hodge structure on $H^*(M)$.
\end{thm}
Combining the results of Nomizu (\cref{Nomizu}), Morgan (\cref{Morgan}) and Deligne (\cref{Del}) we have the following corollary.
\begin{cor}\label{exNomizu}
Let $M$ be a smooth quasi-projective variety. Assume that $M$ is diffeomorphic to a non-compact nilmanifold $\Gamma\bs N\times \bb{R}^m$. Let $\fr{n}$ be the Lie algebra of $N$. Then $\fr{n}_{\bb{C}}$ admits the bigrading $\fr{n}_{\bb{C}}=\bo_{p,q\leq 0,p+q\leq -1}\fr{n}_{p,q}$ of a mixed Hodge structure such that for the induced bigrading $H^*(\fr{n}_{\bb{C}})=\bo_{p,q}H_{p,q}^*(\fr{n}_{\bb{C}})$, the following conditions hold:
\ali{
&\bullet\ H^j(\fr{n}_{\bb{C}})=\bo_{(p,q)\in I_j}H_{p,q}^j(\fr{n}_{\bb{C}})\\
&\bullet\ \overline{H_{p,q}^j(\fr{n}_{\bb{C}})}\equiv H_{q,p}^j(\fr{n}_{\bb{C}}) \mod \bo_{s+t<p+q}H_{s,t}^j(\fr{n}_{\bb{C}})
}
for all $0\leq j\leq \dim\fr{n}_{\bb{C}}$ where $I_j$ is given by
\ali{
I_j:=\{ (p,q)\in \bb{Z}^2\mid j\leq p+q\leq 2j,0\leq p,q\leq j\}.
}
\end{cor}
As a consequence, we have the following.
\begin{cor}
Let $N$ be a simply-connected nilpotent real Lie group. If the complexification $\fr{n}_{\bb{C}}$ of the Lie algebra $\fr{n}$ of $N$ does not admit a bigrading in \cref{exNomizu}, then the non-compact nilmanifold $M:=\Gamma\bs N\times \bb{R}^m$ is not a smooth quasi-projective variety for any lattice $\Gamma$ in $N$ and integer $m>0$.  
\end{cor}
\begin{eg}\label{eg}
Let $\fr{n}=\ang{X_1,...,X_n}$ be the Lie algebra such that $[X_1,X_i]=X_{i+1}$ for $2\leq i\leq n-1$ and $[X_i,X_j]=0$ for other $i<j$. Let $N$ be a simply-connected nilpotent Lie group whose Lie algebra is isomorphic to $\fr{n}$. For every lattice $\Gamma$ in $N$, by \cite[Corollary3.1]{Shimoji}, we see that $M:=\Gamma\bs N\times \bb{R}^{m}$ is smooth quasi-projective variety for some $m>0$ if and only if $n=2,3$. Structures of quasi-projective varieties are given by 
\ali{
M=\bb{Z}^2\bs \bb{R}^2\times \bb{R}^{2m}
}   
and 
\ali{
M=\Gamma\bs N\times \bb{R}^{2m+1}=(\Gamma\bs N\times\bb{R})\times\bb{R}^{2m}
}
where $\Gamma\bs N\times\bb{R}$ is a principal $\bb{C}^*$-bundle over a torus $T^2$.
\end{eg}
\begin{eg}\label{eg2}
Let $\fr{n}_{\bb{C}}=\bb{C}^n$. We define $\fr{n}_{-1,-1}:=\bb{C}^n$. Then we have 
$H^p(\bb{C}^n)=H_{p,p}^p(\bb{C}^n)$ Therefore $\bb{C}^n$ admits a bigrading in \cref{exNomizu}.
\end{eg}
\begin{eg}\label{eg3}
Let $\fr{n}_{\bb{C}}=\fr{n}_{2k+1}(\bb{C})$ for $k\geq 1$. Let $X_1,\dots,X_k,\overline{X_1},\dots, \overline{X_k},Z$ be a basis of $\fr{n}_{2k+1}(\bb{C})$ such that $[X_i,\overline{X_i}]=Z$ for $1\leq i\leq k$. We define $\fr{n}_{-1,0}:=\ang{X_1,\dots,X_k}$, $\fr{n}_{0,-1}:=\ang{\overline{X_1},\dots,\overline{X_k}}$ and $\fr{n}_{-1,-1}:=\ang{Z}$. Then the bigrading satisfies the condition in \cref{exNomizu}. For example, we consider a bigrading on $\fr{n}_3(\bb{C})$. The bigradings on the cohomologies of $\fr{n}_3(\bb{C})$ are given by
\ali{
H^1(\fr{n}_{\bb{C}})&=\ang{[x_1]}_{1,0}\o\ang{[y_1]}_{0,1},\\
H^2(\fr{n}_{\bb{C}})&=\ang{[x_1\w z]}_{2,1}\o\ang{[y_1\w z]}_{1,2},\\
H^3(\fr{n}_{\bb{C}})&=\ang{[x_1\w y_1\w z]}_{2,2}.
}
Moreover we see that the bigradings on $\fr{n}_{2k+1}(\bb{C})$ as above satisfies the condition in \cref{exNomizu} for other $k$.
\end{eg}
\begin{eg}\label{eg4}
Let $\fr{n}_7^{142}$ and $\fr{n}_7^{143}$ be the Lie algebras given by
\ali{
\fr{n}_7^{142}&=\ang{X_1,\dots,X_7}: [X_1,X_i]=X_{i-1}\text{ for i=3,5,7 },[X_3,X_5]=X_4,[X_5,X_7]=X_2,\\
\fr{n}_7^{143}&=\ang{X_1,\dots,X_7}: [X_1,X_i]=X_{i-1}\text{ for i=3,5,7 },[X_3,X_5]=X_6,[X_5,X_7]=X_2.
}
By the classification of seven-dimensional nilpotent Lie algebras \cite{Nil}, we have isomorphisms $\fr{n}_7^{142}\cong (37D)$ and $\fr{n}_7^{143}\cong (37B)$ where $(37D)$ and $(37B)$ are given by
\ali{
(37D)&=\ang{e_1,\dots,e_7}:[e_1,e_2]=e_5=[e_3,e_4],[e_1,e_3]=e_6,[e_2,e_4]=e_7,\\
(37B)&=\ang{e_1,\dots,e_7}:[e_1,e_2]=e_5,[e_2,e_3]=e_6,[e_3,e_4]=e_7
}
where the basis transformations are given by
\ali{
(37D)&:e_1:=-X_7+i(X_1-X_3),e_2:=X_3-iX_5,e_3:=\overline{e_1},e_4:=\overline{e_2},e_5:=-2iX_6,e_6:=2iX_2,e_7:=2iX_4,\\
(37B)&:e_1:=X_3+iX_7,e_2:=X_1-iX_5,e_3:=\overline{e_2},e_4:=\overline{e_1},e_5:=-2(X_2+iX_6),e_6:=2(X_2-iX_6),e_7:=2iX_4.
}
By the same argument as in \cref{eg2} and \cref{eg4}, we see that the Lie algebras $(37D)$ and $(37B)$ satisfy the condition of \cref{exNomizu}. One can define bigradings on $(37D)$ and $(37B)$ by
\ali{
(37D)=\ang{e_1,e_2}_{-1,0}\o\ang{e_3,e_4}_{0,-1}\o\ang{e_5,e_6,e_7}_{-1,-1},\\
(37B)=\ang{e_1,e_3}_{-1,0}\o\ang{e_2,e_4}_{0,-1}\o\ang{e_5,e_6,e_7}_{-1,-1}.
}
\end{eg}
\section{Main Results}
In this section, we consider the following problem.
\begin{prob}
When is a smooth quasi-projective variety diffeomorphic to a non-compact nilmanifold ?
\end{prob}
We first consider the case where the first cohomology of $M$ reflects only the information of divisors.
\begin{thm}\label{abel}
Let $M=V\setminus D$ be a smooth quasi-projective variety. Assume that $M$ is diffeomorphic to a non-compact nilmanifold $\Gamma\bs N\times \bb{R}^m$. If $H^1(V)=\0$ for the compactification $V$ of $M$, then $M$ is diffeomorphic to a trivial bundle $T^n\times \bb{R}^m$ over a torus $T^n$ where $n=\dim N$. In particular, $M$ is diffeomorphic to a trivial bundle $T^{b_1(M)}\times \bb{R}^m$ over a $b_1(M)$-dimensional torus.\end{thm}
\begin{proof}
Let $\fr{n}$ be the Lie algebra of $N$. By the isomorphism 
\ali{
H^1(M)_{\bb{C}}=H_{1,1}^1(M)\cong H^1(\fr{n}_{\bb{C}})
}
and \cref{exNomizu}, we have a bigrading on $\fr{n}_{\bb{C}}$
\ali{
\fr{n}_{\bb{C}}=\fr{n}_{-1,-1}\o\fr{n}_{-i_1,-i_1}\o\cdots\o\fr{n}_{-i_k,-i_k}
}
and the induced bigrading 
\ali{
H^1(\fr{n}_{\bb{C}})&=H_{1,1}^1(\fr{n}_{\bb{C}})=\fr{n}_{-1,-1}^*,\\
H^p(\fr{n}_{\bb{C}})&=H_{p,p}^p(\fr{n}_{\bb{C}})
}
for all $p\geq 2$. Let $x_1^j,\dots,x_{m_j}^j$ be a basis of $\fr{n}_{-i_j,-i_j}^*$ where $m_j=\dim\fr{n}_{-i_j,-i_j}^*$ and $\fr{n}_{-i_0,-i_0}^*:=\fr{n}_{-1,-1}^*$. Since $\fr{n}_{\bb{C}}$ is a nilpotent Lie algebra, the cohomology class 
\ali{
[x_1^1\w\cdots\w x_{m_1}^1\w\cdots\w x_{m_k}^k]
}
does not vanish in $H_{n,n}^{\dim\fr{n}_{\bb{C}}}(\fr{n}_{\bb{C}})\subset H^{\dim\fr{n}_{\bb{C}}}(\fr{n}_{\bb{C}})=H_{\dim\fr{n}_{\bb{C}},\dim\fr{n}_{\bb{C}}}^{\dim\fr{n}_{\bb{C}}}(\fr{n}_{\bb{C}})$. Thus we have
\ali{
\sum_{j=0}^km_j&=\dim\fr{n}_{\bb{C}}=n=\sum_{j=0}^ki_jm_j\\
&=\sum_{j=0}^km_j+\sum_{j=0}^k(i_j-1)m_j
}
and hence
\ali{
0=\sum_{j=0}^k(i_j-1)m_j=\sum_{j=1}^k(i_j-1)m_j.
}
Since $i_j-1>0$ for all $1\leq j\leq k$, we have $m_j=0$ for all $1\leq j\leq k$. Thus $N$ is abelian Lie group. Therefore there exists an integer $n\geq 1$ such that $\fr{n}_{\bb{C}}$ is isomorphic to an abelian Lie algebra $\bb{C}^n$ and $M=\bb{Z}^n\bs\bb{R}^n\times \bb{R}^m$. 
\end{proof}
We next show the following. The idea of the proof is the same as \cref{abel}.
\begin{thm}\label{abel2step}
Let $M$ be a smooth quasi-projective variety. Assume that $M$ is diffeomorphic to a non-compact nilmanifold $\Gamma\bs N\times \bb{R}^m$. Then $M$ is diffeomorphic to a trivial bundle $T^{b_1(M)}\times\bb{R}^m$ over a $b_1(M)$-dimensional torus $T^{b_1(M)}$ or a trivial bundle $E\times \bb{R}^m$ such that $E$ is a torus bundle $E\rightarrow T^{b_1(M)}$ over a torus $T^{b_1(M)}$. In particular, the fundamental group $\pi_1(M,x)$ is torsion-free nilpotent and the nilpotency class is at most two. 
\end{thm}
\begin{proof}
Let $\fr{n}$ be the Lie algebra of $N$. By \cref{exNomizu}, there exists a bigrading $\fr{n}_{\bb{C}}=\bo_{p+q\leq -1,p,q\leq 0}\fr{n}_{p,q}$ such that for the induced bigrading $H^*(\fr{n}_{\bb{C}})=\bo_{p+q\geq 1,p,q\geq 1}H_{p,q}^*(\fr{n}_{\bb{C}})$, we have
\ali{
H^j(\fr{n}_{\bb{C}})&=\bo_{(p,q)\in I_j}H_{p,q}^j(\fr{n}_{\bb{C}}),\\
\overline{H_{p,q}^j(\fr{n}_{\bb{C}})}&\equiv H_{q,p}^j(\fr{n}_{\bb{C}})\mod \bo_{s+t<p+q}H_{s,t}^j(\fr{n}_{\bb{C}})
}  
for all $j\geq 1$ where $I_j$ is as in \cref{exNomizu}. Since $\fr{n}_{\bb{C}}$ is a nilpotent Lie algebra, there exists a $(\dim\fr{n}_{\bb{C}})$-form $\omega$ such that 
\ali{
0\neq [\omega]\in H_{s,t}^{\dim\fr{n}_{\bb{C}}}(\fr{n}_{\bb{C}}).
}
Since 
\ali{
s=\sum_{\fr{n}_{p,q}\neq \0}(-p)\dim\fr{n}_{p,q}\leq \dim\fr{n}_{\bb{C}}=\sum_{\fr{n}_{p,q}\neq \0}\dim\fr{n}_{p,q}
}
and
\ali{
t=\sum_{\fr{n}_{p,q}\neq \0}(-q)\dim\fr{n}_{p,q}\leq \dim\fr{n}_{\bb{C}}=\sum_{\fr{n}_{p,q}\neq \0}\dim\fr{n}_{p,q},
} 
we have
\ali{
\sum_{\fr{n}_{p,q}\neq \0}(-p-1)\dim\fr{n}_{p,q}\leq 0
}
and
\ali{
\sum_{\fr{n}_{p,q}\neq \0}(-q-1)\dim\fr{n}_{p,q}\leq 0.
}
Thus $\fr{n}_{p,q}=\0$ if $(p,q)\neq (-1,0),(0,-1)$ and $(-1,-1)$. Therefore $\fr{n}_{\bb{C}}$ is abelian or $2$-step nilpotent. Since the Lie algebra $\fr{n}$ of $N$ is $2$-step nilpotent or abelian, the nilmanifold $\Gamma\bs N$ is diffeomorphic to a torus $T^{b_1(M)}$, or a bundle $E\rightarrow T^{b_1(M)}$ over a torus $T^{b_1(M)}$ (\cite{PS61}). Therefore \cref{abel2step} holds.
\end{proof}
Thus it suffices to study a $2$-step nilpotent Lie algebra $\fr{n}$ such that the complexification $\fr{n}_{\bb{C}}$ admits a bigrading $\fr{n}_{\bb{C}}=\fr{n}_{-1,0}\o\fr{n}_{0,-1}\o\fr{n}_{-1,-1}$ such that 
\ali{
C^1\fr{n}_{\bb{C}}=\fr{n}_{-1,-1}
}
and
\ali{
H^1(\fr{n}_{\bb{C}})=H_{1,0}^1(\fr{n}_{\bb{C}})\o H_{0,1}^1(\fr{n}_{\bb{C}})\text{ with }\overline{H_{1,0}^1(\fr{n}_{\bb{C}})}=H_{0,1}^1(\fr{n}_{\bb{C}}).
}
As a consequence, we have the following corollary.
\begin{cor}\label{3stepNG}
Let $N$ be a simply-connected nilpotent Lie group. Assume that the Lie algebra $\fr{n}=\fr{n}_{\bb{Q}}\otimes\bb{R}$ of $N$ is $(3\leq )t$-step nilpotent or, $2$-step and not isomorphic to $\tilde{\fr{n}}\o \bb{R}^n$ for some $2$-step nilpotent Lie algebra $\tilde{\fr{n}}$ and abelian Lie algebra $\bb{R}^n$ such that 
\ali{
\dim H^1(\fr{n})\notin 2\bb{Z}.
}
Then $M=\Gamma\bs N\times \bb{R}^m$ is not a smooth quasi-projective variety for any lattice $\Gamma$ in $N$ and integer $m>0$.
\end{cor}
\section{The case of $\dim N\leq 8$}
In this section, we consider whether a non-compact nilmanifold is diffeomorphic to a smooth quasi-projective variety. First, by combining the results \cite[Theorem 1.3, Theorem1.4]{Shimoji} and \cref{3stepNG}, we can easily determine the Lie algebras of dimension up to $7$ such that corresponding non-compact nilmanifolds may be diffeomorphic to smooth quasi-projective varieties. Moreover, bigradings which satisfy the conditions in \cref{exNomizu} are given in \cref{eg2},\cref{eg3} and \cref{eg4}.
\begin{cor}\label{upto7}
Let $M=\Gamma \bs N\times \bb{R}^m$ be a non-compact nilmanifold. Assume that  $M$ is diffeomorphic to a smooth quasi-projective variety. If $\dim N\leq 7$, then the complexification $\fr{n}_{\bb{C}}$ of the Lie algebra $\fr{n}$ of $N$ is isomorphic to one of the following:
\ali{
&\bb{C}&\text{ if }b_1(\Gamma\bs N)=1\\
&\bb{C}^2\text{ or }\fr{n}_3(\bb{C})&\text{ if }b_1(\Gamma\bs N)=2\\
&\bb{C}^3\text{ or }\fr{n}_3(\bb{C})\o\bb{C}&\text{ if }b_1(\Gamma\bs N)=3\\
&\bb{C}^4,\fr{n}_3(\bb{C})\o\bb{C}^2,\fr{n}_5(\bb{C}),\fr{n}_3(\bb{C})\o\fr{n}_3(\bb{C}),\fr{n}_7^{142}\text{ or }\fr{n}_7^{143}&\text{ if }b_1(\Gamma\bs N)=4\\
&\bb{C}^5,\fr{n}_3(\bb{C})\o\fr{n}_3(\bb{C})\o\bb{C},\fr{n}_5(\bb{C})\text{ or }\fr{n}_3(\bb{C})\o\bb{C}^3&\text{ if }b_1(\Gamma\bs N)=5\\
&\bb{C}^6,\fr{n}_3(\bb{C})\o\bb{C}^4,\fr{n}_5(\bb{C})\o\bb{C}^2\text{ or }\fr{n}_7(\bb{C})&\text{ if }b_1(\Gamma\bs N)=6\\
&\bb{C}^7&\text{ if }b_1(\Gamma\bs N)=7
}
where $\fr{n}_7^{142}$ and $\fr{n}_7^{143}$ are given by
\ali{
\fr{n}_7^{142}=\ang{X_1,\dots,X_7}:&[X_1,X_i]=X_{i-1}\text{ for }i=3,5,7\\
&[X_3,X_5]=X_4,[X_5,X_7]=X_2,\\
\fr{n}_7^{143}=\ang{X_1,\dots,X_7}:&[X_1,X_i]=X_{i-1}\text{ for }i=3,5,7\\
&[X_3,X_5]=X_6,[X_5,X_7]=X_2.
}
\end{cor}
By using the classification of $2$-step nilpotent Lie algebras of dimension $8$ in \cite{YanDeng} and \cref{3stepNG}, we can determine the nilpotent Lie groups such that  constructed non-compact nilmanifolds may be diffeomorphic to smooth quasi-projective varieties.  
\begin{thm}\label{dim8}
Let $M=\Gamma \bs N\times \bb{R}^m$ be a non-compact nilmanifold. Assume that  $M$ is diffeomorphic to a smooth quasi-projective variety and the Lie algebra $\fr{n}$ of $N$ is indecomposable. If $\dim N=8$, then the complexification of the Lie algebra is isomorphic to one of the following:
\ali{
&N_1^{8,4}&\text{ if }b_1(\Gamma\bs N)=4,\\
&N_i^{8,2} (1\leq i\leq 5)&\text{ if }b_1(\Gamma\bs N)=6,\\
&\bb{C}^8&\text{ if }b_1(\Gamma\bs N)=8
}
where the Lie algebras $N_1^{8,4},N_i^{8,2}$ for $1\leq i\leq 5$ are given by
\ali{
N_1^{8,4}&=\ang{X_1,X_2,\overline{X_1},\overline{X_2},Z_1,Z_2,Z_3,Z_4}:[X_1,\overline{X_1}]=Z_1,[X_1,\overline{X_2}]=Z_2,[X_2,\overline{X_1}]=Z_3,[X_2,\overline{X_2}]=Z_4,\\
N_1^{8,2}&=\ang{x_1,x_2,x_3,x_4,x_5,x_6,x_7,x_8}:[x_1,x_2]=x_7,[x_3,x_4]=x_8,[x_5,x_6]=x_7+x_8,\\
N_2^{8,2}&=\ang{x_1,x_2,x_3,x_4,x_5,x_6,x_7,x_8}:[x_1,x_2]=x_7=[x_4,x_5],[x_1,x_3]=x_8=[x_4,x_6],\\
N_3^{8,2}&=\ang{x_1,x_2,x_3,x_4,x_5,x_6,x_7,x_8}:[x_1,x_2]=x_7=[x_4,x_5],[x_3,x_4]=x_8=[x_5,x_6],\\
N_4^{8,2}&=\ang{x_1,x_2,x_3,x_4,x_5,x_6,x_7,x_8}:[x_1,x_2]=x_7=[x_3,x_4]=[x_5,x_6],[x_4,x_5]=x_8,\\
N_5^{8,2}&=\ang{x_1,x_2,x_3,x_4,x_5,x_6,x_7,x_8}:[x_1,x_2]=x_7=[x_3,x_4]=[x_5,x_6],[x_4,x_5]=x_8=[x_2,x_3].
}
\end{thm}
\begin{proof}
By \cite[Theorem1.1]{Shimoji}, we have 
\ali{
\dim H^1(\fr{n}_{\bb{C}})\geq 4.
}
We first assume $\dim H^1(\fr{n}_{\bb{C}})= 4$. Then we see that $\fr{n}_{\bb{C}}$ is isomorphic to the following Lie algebra $N_1^{8,4}$ in \cite{YanDeng}:
\ali{
N_1^{8,4}&=\ang{X_1,X_2,\overline{X_1},\overline{X_2},Z_1,Z_2,Z_3,Z_4}:[X_1,\overline{X_1}]=Z_1,[X_1,\overline{X_2}]=Z_2,[X_2,\overline{X_1}]=Z_3,[X_2,\overline{X_2}]=Z_4.
}
Since \cref{3stepNG}, we next assume $\dim H^1(\fr{n}_{\bb{C}})= 6$. By \cite{YanDeng}, then $\fr{n}_{\bb{C}}$ is isomorphic to one of the following Lie algebras:
\ali{
N_1^{8,2}&=\ang{x_1,x_2,x_3,x_4,x_5,x_6,x_7,x_8}:[x_1,x_2]=x_7,[x_3,x_4]=x_8,[x_5,x_6]=x_7+x_8,\\
N_2^{8,2}&=\ang{x_1,x_2,x_3,x_4,x_5,x_6,x_7,x_8}:[x_1,x_2]=x_7=[x_4,x_5],[x_1,x_3]=x_8=[x_4,x_6],\\
N_3^{8,2}&=\ang{x_1,x_2,x_3,x_4,x_5,x_6,x_7,x_8}:[x_1,x_2]=x_7=[x_4,x_5],[x_3,x_4]=x_8=[x_5,x_6],\\
N_4^{8,2}&=\ang{x_1,x_2,x_3,x_4,x_5,x_6,x_7,x_8}:[x_1,x_2]=x_7=[x_3,x_4]=[x_5,x_6],[x_4,x_5]=x_8,\\
N_5^{8,2}&=\ang{x_1,x_2,x_3,x_4,x_5,x_6,x_7,x_8}:[x_1,x_2]=x_7=[x_3,x_4]=[x_5,x_6],[x_4,x_5]=x_8=[x_2,x_3].
}
\end{proof}
\begin{rem}
We can show the existence of bigradings on the above Lie algebras. For example, we define the bigrading on $N_1^{8,4}$ by
\ali{
N_1^{8,4}=\ang{X_1,X_2}_{-1,0}\o\ang{\overline{X_1},\overline{X_2}}_{0,-1}\o\ang{Z_1,Z_2,Z_3,Z_4}_{-1,-1}.
}
Then the induced bigrading on the cohomology $H^*(N_i^{8,2})=\bo_{p,q}H_{p,q}^*(N_i^{8,2})$ satisfy the condition in \cref{exNomizu}. Moreover, we define the bigrading on $N_i^{8,2}$ by 
\ali{
N_i^{8,2}=\ang{A_1,A_2,A_3}_{-1,0}\o\ang{\overline{A_1},\overline{A_2},\overline{A_3}}_{0,-1}\o\ang{B_1,B_2}_{-1,-1}
}
where the basis transformations are given by
\ali{
&A_1:=x_1+ix_2,A_2:=x_3+ix_4,A_3:=x_5+ix_6,B_1:=-2ix_7,B_2:=-2ix_8&\text{ if }\fr{n}_{\bb{C}}\cong N_1^{8,2},\\
&A_1:=x_1+ix_4,A_2:=x_3+ix_6,A_3:=-x_5+ix_2,B_1:=-2ix_7,B_2:=-2ix_8&\text{ if }\fr{n}_{\bb{C}}\cong N_2^{8,2},\\
&A_1:=x_1+ix_2,A_2:=x_3+ix_6,A_3:=x_5+ix_4,B_1:=-2ix_7,B_2:=-2ix_8&\text{ if }\fr{n}_{\bb{C}}\cong N_3^{8,2}\text{ or }N_4^{8,2},
}
\ali{
A_1:=x_1+ix_6,A_2:=x_3+i(x_4+x_6),A_3:=x_5+i(x_2+x_4+x_6),B_1:=-2ix_7,B_2:=-2ix_8\text{ if }\fr{n}_{\bb{C}}\cong N_5^{8,2}
}
Then the induced bigrading on the cohomology $H^*(N_i^{8,2})=\bo_{p,q}H_{p,q}^*(N_i^{8,2})$ satisfy the condition in \cref{exNomizu}.
\end{rem}
\section{Comparing with our previous work}\label{Prework}
In \cite{Shimoji}, we considered the following problem.
\begin{prob}\label{Campana}
Let $X$ be a smooth quasi-projective variety over $\bb{C}$. Assume that the (topological) fundamental group $\pi_1(X,x)$ is torsion-free nilpotent. Is $\pi_1(X,x)$ abelian or $2$-step nilpotent?
\end{prob}
By using the Morgan's result for Lie algebras in \cite[Theorem (9.4)]{Morgan}, we showed the following result.
\begin{thm}[\cite{Shimoji}]
Let $X$ be a smooth quasi-projective variety. Assume that $\pi_1(X,x)$ is torsion-free nilpotent. If $rk(\pi_1(X,x))\leq 7$, then $\pi_1(X,x)$ is abelian or $2$-step nilpotent.
\end{thm}
Moreover, we determined the Lie groups of dimension up to $7$ whose lattices may be isomorphic to the fundamental groups of smooth quasi-projective varieties (see \cite[Theorem 1.3 and 1.4]{Shimoji}). However, from rank $8$, we cannot determine the sat classes of the fundamental groups of smooth quasi-projective varieties. In fact, there exists a finitely generated torsion-free $3$-step nilpotent group such that it satisfies the Morgan's result which is used in \cite{Shimoji}. The example is a lattice in the following Lie group 
\ali{
N(\bb{R}):=\left \{\begin{pmatrix}
1&0&y_2&y_1&-z_2&-z_1&0&\frac{1}{2}(y_1+y_2)&3b\\
0&1&x_2&x_1&0&\frac{1}{2}(x_1+x_2)&-z_2&-z_1&3a\\
0&0&1&0&x_1&x_2&-y_1&-y_2&2z_2\\
0&0&0&1&x_2&x_1&-y_2&-y_1&2z_1\\
0&0&0&0&1&0&0&0&y_2\\
0&0&0&0&0&1&0&0&y_1\\
0&0&0&0&0&0&1&0&x_2\\
0&0&0&0&0&0&0&1&x_1\\
0&0&0&0&0&0&0&0&1\\
\end{pmatrix}
;
x_1,x_2,y_1,y_2,z_1,z_2,a,b\in\bb{R}\right \}
.
}
Moreover, the Lie algebra of $N(\bb{R})$ is given by
\ali{
\fr{g}:=\ang{X_1,X_2,Y_1,Y_2,Z_1,Z_2,A,B}:&[X_1,Y_1]=Z_1=[X_2,Y_2], [X_2,Y_1]=Z_2=[X_1,Y_2],\\
&[X_1,Z_1]=A=[X_2,Z_2],[Y_1,Z_1]=B=[Y_2,Z_2]
}
and the bigrading on $\fr{g}_{\bb{C}}$ of a mixed Hodge structure is given by
\ali{
\fr{g}_{\bb{C}}&=\ang{A_1,A_2}_{-1,0}\o\ang{\overline{A_1},\overline{A_2}}_{0,-1}\o\ang{B_1,B_2}_{-1,-1}\o\ang{C_1}_{-2,-1}\o\ang{\overline{C_1}}_{-1,-2}
}
where $A_1:=X_1+iY_1,A_2:=X_2+iY_2,B_1:=-2iZ_1,B_2:=-2iZ_2$ and $C_1:=-2i(A+iB)$. This bigrading satisfies conditions which are shown by Morgan \cite[Theorem (9.4)]{Morgan}. In this paper, we focus on the class of the intersection between smooth quasi-projective varieties and non-compact nilmanifolds. Then we can show the restriction of the nilpotency classes of the fundamental groups in arbitrary dimension. Thus, by \cref{abel2step}, we obtain a following corollary.
\begin{cor}
Let $N$ be a simply-connected nilpotent Lie group and $\Gamma$ a lattice in $N$. Assume that the Lie algebra of $N$ is isomorphic to the above Lie algebra $\fr{g}$. If there exists a smooth quasi-projective variety $M$ such that the fundamental group $\pi_1(M,x)$ is isomorphic to $\Gamma$, then $M$ is not diffeomorphic to any non-compact nilmanifold.
\end{cor}

\end{document}